\documentclass{amsart}
\usepackage[mathscr]{eucal}

\usepackage{hyperref}

\allowdisplaybreaks[1]

\newtheorem{thrm}{Theorem}[section]
\newtheorem{lem}[thrm]{Lemma}

\theoremstyle{definition}

\numberwithin{equation}{section}

\newtheorem{problem}[thrm]{Problem}

\newcommand{\labeq}[1]{\label{eq:#1}}
\newcommand{\refeq}[1]{(\ref{eq:#1})}
\newcommand{\labt}[1]{\label{thm:#1}}
\newcommand{\reft}[1]{Theorem~\ref{thm:#1}}

\newcommand{\labl}[1]{\label{lemma:#1}}
\newcommand{\refl}[1]{Lemma~\ref{lemma:#1}}

\newcommand{\lmeas}[1]{\lambda\left( #1 \right)}

\newcommand{\wrt}[1]{\hbox{ w.r.t. }#1}

\newcommand{\floor}[1]{\left\lfloor #1 \right\rfloor}

\newcommand{\NQ}{\mathscr{N}(Q)}
\newcommand{\N}[1]{\mathscr{N}( #1 )}
 
\newcommand{\DNQ}{\mathscr{DN}(Q)}
\newcommand{\DN}[1]{\mathscr{DN}( #1 )}

\newcommand{\IB}{\mathcal{I}_{Q,j}(B)}
\newcommand{\IBn}[1]{\mathcal{I}_{Q,#1}(B)}

\newcommand{\blank}[1]{ }

\author[D. Airey]{Dylan Airey}
\address[D. Airey]{
Department of Mathematics, University of Texas at Austin, 2515 Speedway, Austin, TX 78712-1202, USA}
\email{dylan.airey@utexas.edu}

\author[B. Mance]{Bill Mance}
\address[B. Mance]{Department of Mathematics, University of North Texas, General Academics Building 435, 1155 Union Circle,  \#311430, Denton, TX 76203-5017, USA}
\email{mance@unt.edu}

\begin{document}

\thanks{Research of the authors is partially supported by the U.S. NSF grant DMS-0943870.  The authors  thank Joseph Vandehey for helpful discussions.}

\title[Normal equivalencies for eventually periodic basic sequences]{Normal equivalencies for eventually periodic basic sequences}

\begin{abstract}
W. M. Schmidt, A. D. Pollington, and F. Schweiger have studied when normality with respect to one expansion is equivalent to normality with respect to another expansion. Following in their footsteps, we show that when $Q$ is an eventually periodic basic sequence, that $Q$-normality and $Q$-distribution normality are equivalent to normality in base $b$ where $b$ is dependent on $Q$. We also show that boundedness of the basic sequence is not sufficient for this equivalence.
\end{abstract}

\maketitle

\section{Introduction}


Let $\N{b}$ be the set of {\it normal numbers in base $b$}. We say that two natural numbers $r$ and $s$ are {\it equivalent}, or $r\sim s$, if $\frac{\log r}{\log s}$ is rational. W. M. Schmidt showed the following in \cite{SchmidtRelated}.
\begin{thrm}[W. M. Schmidt, 1960]\labt{Schmidt}
We have $\N{r} = \N{s}$ if and only if $r\sim s$. If $r\not \sim s$, then $\N{r} \setminus \N{s}$ is uncountable.
\end{thrm}
A. D. Pollington strengthened W. M. Schmidt's result in \cite{PollingtonHDNormal}.
\begin{thrm}[A. D. Pollington, 1981]
Given any partition of the numbers $2, 3, \cdots$ into two disjoint classes $R$ and $S$ such that equivalent numbers fall in the same class, the set $\mathcal{N}$ of numbers which are normal to every base from $R$ and to no base from $S$ has Hausdorff dimension 1.
\end{thrm}

A {\it number-theoretic transformation} $T$ is defined by describing its action on at most countably many subsets $I(k) \subset [0,1]$ with $\bigcup I(k) = [0,1]$ and $\lambda(I(k) \cap I(j)) = 0$ if $k \neq j$. A condition is also placed on the Jacobian of $T$ to
 ensure that $T$ is ergodic and that there is a unique $T$-invariant probabilty measure $\mu$ that is absolutely continuous with
 respect to $\lambda$ (this condition can be found in \cite{SchweigerNormalEquiv}). {\it Cylinder sets} $I(k_1, \cdots, k_n)$ are
 defined recursively by $I(k_1, \cdots, k_n) = T(k_1)^{-1} I(k_2, \cdots, k_n)$. We say that $x \in [0,1]$ is {\it $T$-normal} if for
 every cylinder set $E = I(k_1, \cdots, k_n)$ we have 
$$
\lim_{N \to \infty} \frac{A_N(E, (T^n x))}{N} = \mu(E)
$$
where $A_N(E, (x_n)) = \# \{n\leq N : x_n \in E \}$ for a sequence of real numbers $(x_n)$ and an interval $E$. This is equivalent to the condition that the sequence $(T^n x)$ is $\mu$-uniformly distributed mod 1.
For a number-theoretic transformation $T$, let $\N{T}$ be the set of normal numbers with respect to $T$. Generalizations of \reft{Schmidt} have been considered for $T$-normal numbers. 
 J. Vandehey \cite{VandeheyNormalEquiv} presented a corrected proof  of a result announced by F. Schweiger in \cite{SchweigerNormalEquiv}.
\begin{thrm}[F. Schweiger, 1969 and J. Vandehey, 2014]\labt{Vandehey}
 For any ergodic number-theoretic transformation $\N{T} = \N{T^n}$ for all $n \geq 1$.
\end{thrm}
J. Vandehey also showed the equivalence of normality with respect to the regular continued fraction expansion and with respect to the odd continued fraction expansion.
C. Kraaikamp and H. Nakada \cite{KraaikampNakada} answered a conjecture of F. Schweiger \cite{SchweigerNormalEquiv} on normal numbers.
\begin{thrm}[C. Kraaikamp and H. Nakada, 2001]
There exist number-theoretic transformations $S$ and $T$ such that $\N{S} = \N{T}$ and such that there are no positive integers $m$ and $n$ with $S^m = T^n$.
\end{thrm}

We will extend W. M. Schmidt's result to Cantor series expansions, a generalization of $b$-ary expansions. 
The study of normal numbers and other statistical properties of real numbers with respect to large classes of Cantor series expansions was  first done by P. Erd\H{o}s and A. R\'{e}nyi in \cite{ErdosRenyiConvergent} and \cite{ErdosRenyiFurther} and by A. R\'{e}nyi in \cite{RenyiProbability}, \cite{Renyi}, and \cite{RenyiSurvey} and by P. Tur\'{a}n in \cite{Turan}.

The $Q$-Cantor series expansions, first studied by G. Cantor in \cite{Cantor},
are a natural generalization of the $b$-ary expansions.\footnote{G. Cantor's motivation to study the Cantor series expansions was to extend the well known proof of the irrationality of the number $e=\sum 1/n!$ to a larger class of numbers.  Results along these lines may be found in the monograph of J. Galambos \cite{Galambos}. } 
A basic sequence is a sequence of integers greater than or equal to $2$.
Given a basic sequence $Q=(q_n)_{n=1}^{\infty}$, the {\it $Q$-Cantor series expansion} of a real number $x$  is the (unique)\footnote{Uniqueness can be proven in the same way as for the $b$-ary expansions.} expansion of the form
\begin{equation} \labeq{cseries}
x=E_0+\sum_{n=1}^{\infty} \frac {E_n} {q_1 q_2 \cdots q_n}
\end{equation}
where $E_0=\floor{x}$ and $E_n$ is in $\{0,1,\cdots,q_n-1\}$ for $n\geq 1$ with $E_n \neq q_n-1$ infinitely often. We abbreviate \refeq{cseries} with the notation $x=E_0.E_1E_2E_3\cdots$ w.r.t. $Q$.


For a basic sequence $Q=(q_n)$, a block $B=(b_1, b_2, \cdots, b_k)$, and a natural number $j$, define
$$
\IB = \begin{cases}
1 &\text{ if } b_1 < q_j, b_2 < q_{j+1}, \cdots, b_k < q_{j+k-1} \\
0 &\text{ otherwise}
\end{cases}
$$
and let
$$
Q_n(B)=\sum_{j=1}^n \frac {\IB} {q_j q_{j+1} \cdots q_{j+k-1}}.
$$
We also define
$$
T_{Q,n}(x) = q_n q_{n-1} \cdots q_1 x \pmod{1}.
$$


A. R\'enyi \cite{Renyi} defined a real number $x$ to be {\it normal} with respect to $Q$ if for all blocks $B$ of length $1$,
\begin{equation}\labeq{rnormal}
\lim_{n \rightarrow \infty} \frac {N_n^Q (B,x)}{\sum_{i=1}^n 1/q_i}=1,
\end{equation}
where $N_n^Q(B,x)$ is the number of occurences of the block $B$ in the sequence $(E_i)_{i=1}^n$ of the first $n$ digits in the $Q$-Cantor series expansion of $x$. If $q_n=b$ for all $n$ and we restrict $B$ to consist of only digits less than $b$, then \refeq{rnormal} is equivalent to {\it simple normality in base $b$}, but not equivalent to normality in base $b$. 

 A great deal of information about the $b$-ary expansion of a real number $x$ may be obtained by studying the distributional properties of the sequence $(b^n x)_{n=0}^\infty$.  For example, it is well known that a real number $x$ is normal in base $b$ if and only if the sequence $(b^n x)$ is uniformly distributed mod $1$. 

A real number $x$  is {\it $Q$-normal} if\footnote{We choose to take a slightly different definition for $Q$-normality than is used elsewhere in the literature. Our definition is more appropriate for bounded basic sequences.} for all blocks $B$ such that $\lim_{n \to \infty} Q_n(B) = \infty$
$$
\lim_{n \rightarrow \infty} \frac {N_n^Q (B,x)} {Q_n(B)}=1.
$$
Let $\NQ$ be the set of $Q$-normal numbers.
A real number~$x$ is {\it $Q$-distribution normal} if
the sequence $(T_{Q,n}(x))_{n=0}^\infty$ is uniformly distributed mod $1$.  Let $\DNQ$ be the set of $Q$-distribution normal numbers.


Note that in base~$b$, where $q_n=b$ for all $n$,
 the corresponding notions of $Q$-normality and $Q$-distribution normality are equivalent. This equivalence
is fundamental in the study of normality in base $b$. Note that $Q$-normality and $Q$-distribution normality are not equivalent for all basic sequences \cite{ppq1}.


For an eventually periodic basic sequence $Q$, we define a {\it period} of $Q$ to be a tuple $(c_1, c_2, \cdots, c_m)$ such that $Q$ can be written as 
$$
(d_1, d_2, \cdots, d_k, \overline{c_1, c_2, \cdots, c_m}).
$$
The main goal of this paper will be to prove the following theorem.
\begin{thrm}\labt{main}
Let $Q$ be an eventually periodic sequence with period $(c_1, c_2, \cdots, c_m)$. Set $b = \prod_{i=1}^m c_i$ and suppose that $g \sim b$. Then for any real number $x$, the following are equivalent:
\begin{enumerate}
\item $x$ is normal in base $g$
\item $x$ is $Q$-normal
\item $x$ is $Q$-distribution normal
\end{enumerate}
\end{thrm}

Note that \reft{main} is an extension of one direction of Schmidt's result to Cantor series expansions. Some of the ideas in J. Vandehey's proof of \reft{Vandehey} were used in the proof of \reft{main}. We also prove the following theorem which shows some limitations on how much \reft{main} can be generalized.

\begin{thrm}\labt{secondthrm}
For every real number $x$ and integer $g\geq 2$, there exists a basic sequence $Q = (q_n)$ where $q_n = g^{k_n}$ for some sequence of natural numbers $(k_n)$ such that $x \notin \NQ \cup \DNQ$. Thus there exists a basic sequence $Q = (q_n)$ where $q_n = g^{k_n}$ for some sequence of natural numbers $(k_n)$ such that $\N{g} \neq \N{Q}$ and $\N{g} \neq \DN{Q}$.
\end{thrm}

\section{Proofs}

We will need the following theorem due to I. I. Shapiro-Pyatetskii in \cite{PyatetskiiShapiro}.
\begin{thrm}[I. I. Shapiro-Pyatetskii, 1951]\labt{hotspot}
If there exists a constant $C$ such that for every interval $E \subseteq [0,1]$ we have
$$
\limsup_{n \to \infty} \frac{\# \{ i< n : b^i x \in E\}}{n} \leq C \lambda(E)
$$
then $x$ is normal in base $b$.
\end{thrm}


\begin{lem}\labl{normequiv}
If $Q=(d_1, \cdots, d_k, \overline{c_1, \cdots, c_m})$ and $P = (\overline{c_1, \cdots, c_m})$, then $x$ is $Q$-normal (resp. $Q$-distribution normal) if and only if $d_1 \cdots d_k x$ is $P$-normal (resp. $P$-distribution normal).
\end{lem}
\begin{proof}
Let $B$ be a block of length $\ell$ such that $\lim_{n \to \infty} Q_n(B) = \infty$. If $x = 0.E_1 E_2 \cdots\wrt{Q}$, then
$$
d_1 \cdots d_k x \pmod{1}= T_{Q,k}(x) = 0.E_{k+1} E_{k+2}\cdots \wrt{P}.
$$ 
 Thus $N_{n}^P(B,d_1 \cdots d_k x) = N_n^Q(B,x)+O(1)$. Furthermore, 
$$
Q_n(B) = P_n(B) + \sum_{i=1}^k \frac{\mathcal{I}_i(B)}{q_i \cdots q_{i+\ell-1}} = P_n(B)+O(1).
$$ 
Thus
$$
\lim_{n \to \infty} \frac{N_n^P(B,d_1 \cdots d_k x)}{P_n(B)} = \lim_{n \to \infty} \frac{N_n^Q(B,x)+O(1)}{Q_n(B)+O(1)}= \lim_{n \to \infty} \frac{N_n^Q(B,x)}{Q_n(B)},
$$
so $d_1 \cdots d_k x$ is $P$-normal if and only if $x$ is $Q$-normal.

Note that $T_{P,n}(d_1 \cdots d_k x) = T_{Q,n+k}(x)$. Thus, the sequence $(T_{Q,n}(x))$ is u.d. mod 1 if and only if $(T_{P,n}(d_1 \cdots d_k x))$ is u.d. mod 1.
\end{proof}

Define 
$$
J_r(B) = \left [ \frac{b_1 c_{r+1} \cdots c_{r+k-1} + \cdots + b_k}{c_r c_{r+1} \cdots c_{r+k-1}} , \frac{b_1 c_{r+1} \cdots c_{r+k-1} + \cdots + b_k + 1 }{c_r c_{r+1} \cdots c_{r+k-1}} \right )
$$ 
if $\ \IBn{r} = 1$ and $\emptyset$ if $\IBn{r}=0$ for $0 \leq r < m$.
We may now prove \reft{main}.

Suppose that $Q$ can be written as $(d_1, d_2, \cdots, d_k, \overline{c_1, c_2, \cdots, c_m})$. By \refl{normequiv}, $x$ is $Q$-normal (resp. $Q$-distribution normal) if and only if $d_1 \cdots d_k x$ is $P$-normal (resp. $P$-distribution normal), where 
$P = (\overline{c_1, c_2, \cdots, c_m})$. Similarly, $x$ is normal in base $b$ if and only if $d_1 \cdots d_k x$ is normal in base $b$. 
Thus we need only show equivalence of normality for periodic basic sequences. The same argument holds for $Q$-distribution normality. By \reft{Schmidt} we have that $\N{g} = \N{b}$, so we only need to show that $\N{Q} = \N{b}$.

Thus, we may assume that $Q$ is periodic, with period $(c_1, c_2, \cdots, c_m )$ of minimal length. If $m=1$, then we are in the b-ary case so the statements hold trivially. Thus we may assume that $m>1$.
\begin{itemize}
\item  $(1) \implies (3)$.
Suppose that $x$ is normal in base $b$. Then the sequence $(b^n x)$ is u.d. mod 1. Note that the sequence $(T_{Q,n}(x))$ can be decomposed into the subsequences $(T_{Q,mn+r}(x))$ for $0 \leq r < m$, and each subsequence $(T_{Q,mn+r}(x))$ can be rewritten as $(b^n c_1 c_2 \cdots c_r x)$. Since $c_1 c_2 \cdots c_r x$ is an integer multiple of a number that is normal in base $b$, the real number $c_1 c_2 \cdots c_r x$ is normal in base $b$, so $(T_{Q,mn+r}(x))$ is u.d. mod 1. 
Thus, the sequence $(T_{Q,n}(x))$ is u.d. mod 1, so $x$ is $Q$-distribution normal.
\item $(3) \implies (1)$.
Suppose $(T_{Q,n}(x))$ is u.d. mod 1. Then for any interval $E \subseteq [0,1]$, we have 
\begin{align*}
&\limsup_{n \to \infty} \frac{\# \{i \leq n : b^i x \in E \}}{n} \leq \limsup_{n \to \infty} \frac{\# \{i < mn : T_{Q,i}(x) \in E\}}{n} \\
&\leq m \limsup_{n \to \infty} \frac{\#\{i \leq mn : T_{Q,i}(x) \in E\}}{mn} = m \lambda(E).
\end{align*}
Thus by \reft{hotspot}, we have $(b^n x)$ is u.d. mod 1, so $x$ is normal in base~$b$.

\item $(1) \implies (2)$.
Suppose that $x$ is normal in base $b$. Let $B = (b_1, b_2, \cdots, b_k)$ be a block of digits of length $k$. 
The block $B$ occurs at position $n$ in $x$ if and only if
$$
x = \sum_{i=1}^{n-1} \frac{E_i}{c_1 \cdots c_i} + \frac{b_1}{c_1 \cdots c_n} + \frac{b_2}{c_1 \cdots c_{n+1}} + \cdots + \frac{b_k}{c_1 \cdots c_{n+k-1}} + \sum_{i = n+k}^\infty \frac{E_i}{c_1 \cdots c_i}
$$ 
is the $Q$-Cantor series expansion of $x$. That is, if $n \equiv r \mod m$, then 
$
T_{Q,n-1}(x) \in J_r(B).$
Thus
\begin{equation}\labeq{onetwo}
N_n^Q(B,x) = \sum_{r = 0}^{m-1} A_{\floor{n/m}}(J_r(B),(T_{Q,mn+r}(x))) + O(1).
\end{equation}
Since $x$ is normal in base $b$, the sequence $(b^n x)$ is u.d. mod 1. By \refeq{onetwo}
\begin{align*}
\lim_{n \to \infty} \frac{N_n^Q(B,x)}{n} &
= \sum_{r =0}^{m-1}  \lim_{n \to \infty} \frac{\floor{n/m}}{n} \cdot\frac{A_{\floor{n/m}}(J_r(B),(T_{Q,mn+r}(x)))}{\floor{n/m}} \\ 
&= \frac{1}{m} \sum_{r = 0}^{m-1} \lambda(J_r(B)).
\end{align*}
Moreover,
\begin{align*}
&\lim_{n \to \infty} \frac{Q_{n}(B)}{n} = \lim_{n \to \infty} \frac{1}{n} \sum_{i = 0}^{n-1} \frac{\mathcal{I}_i(B)}{c_i c_{i+1} \cdots c_{i+k-1}} \\ 
= &\lim_{n \to \infty} \frac{1}{n} \floor{\frac{n}{m}} \sum_{r = 0}^{m-1} \frac{\mathcal{I}_r(B)}{c_r \cdots c_{r+k-1}}  = \frac{1}{m} \sum_{r = 0}^{m-1} \lambda(J_r(B)).
\end{align*}
Thus we have that $\lim_{n \to \infty} \frac{N_n^Q(B,x)}{Q_n(B)} = 1$, so $x$ is $Q$-normal.

\item $(2) \implies (1)$.
Suppose that $x$ is $Q$-normal. Since $m>1$,  there is an infinite family of blocks $\mathcal{B}$ which can only occur at positions $n \equiv 0 \mod m$. For example, consider blocks of the form 
$$
(c_m -1, c_1 -1, \cdots c_{m-1} - 1,b_{m+1},\cdots,b_k).
$$
Let $B \in \mathcal{B}$. Then for $r \neq 0$, we have that $J_r(B) = \emptyset$, so 
\begin{align*}
&\lim_{n \to \infty} \frac{A_n(J_0(B),(b^i x))}{n} 
= \lim_{n \to \infty} \frac{A_n(J_0(B),(T_{Q,mi}(x)))}{n}\\
=& \lim_{n \to \infty} \frac{N_{mn}^Q(B,x)}{n} 
= \lambda(J_0(B)).
\end{align*}
We will show that the union of all intervals of the form $J_0(B)$ has full measure. To do this, let 
$$
\mathscr{C} = [0,1] \backslash \bigcup_{B \in \mathcal{B}} J_0(B).
$$ 
Let $C_0 = [0,1]$ and define recursively 
$$
C_i = \{ x \in C_{i-1} : E_{mi+r} \neq c_r-1 \text{ for some } 0\leq r <m \}.
$$ 
Then $\lambda(C_i) = (1 - \frac{1}{c_1 \cdots c_{m}}) \lambda(C_{i-1})$, so $\lambda(\bigcap_{i=0}^\infty C_i) = 0$.
 But these $C_i$'s are the sets of real numbers that do not contain the block $\overline{B} = (c_m-1, c_1-1, \cdots, c_{m-1}-1)$ at position $mi$. 
Thus $\mathscr{C} \subset \bigcap_{i=0}^\infty C_i$, so $\lambda(\mathscr{C}) = 0$.

Let $E \subseteq [0,1]$ be an interval. Since $x$ is $Q$-normal, any block $B \in \mathcal{B}$ must occur infinitely often,  so $T_{Q,n}(x) \notin \mathscr{C}$ for all $n$.
 So for all $n$, we have that $A_n(E,(T_{Q,mn}(x))) = A_n(E-\mathscr{C}, (T_{Q,mn}(x)))$.
 Furthermore, since $\mathscr{C}$ is a null set, for each $\epsilon>0$ we can find finitely many blocks $(A_i)_{i=1}^k, (B_j)_{j=1}^\ell\in \mathcal{B}$ such that 
\begin{align*}
&\bigcup_{i=1}^k J_0(A_i) \subset E-\mathscr{C} \subset \bigcup_{j=1}^\ell J_0(B_j) \hbox{ and }\\
&\lambda(E-\mathscr{C})+\epsilon> \lmeas{\bigcup_{j=1}^\ell B_j} \geq \lmeas{\bigcup_{i=1}^k J_0(A_i)} > \lambda(E-\mathscr{C}) - \epsilon.
\end{align*}
But then 
\begin{align*}
\lim_{n \to \infty} \frac{A_n(E-\mathscr{C}, (b^i x))}{n} 
\geq \lim_{n \to \infty} \frac{A_n\left(\bigcup_{i=1}^k J_0(A_i), (b^i x)\right)}{n} \\
= \sum_{n=1}^k \lambda(J_0(A_i)) 
> \lambda(E-\mathscr{C})-\epsilon.
\end{align*}
Similarly, we have that 
$$
\lim_{n \to \infty} \frac{A_n(E-\mathscr{C}, (b^i x))}{n} < \lambda(E-\mathscr{C})+\epsilon.
$$ 
Since $\epsilon$ was arbitrary, 
$$
\lim_{n \to \infty} \frac{A_n(E-\mathscr{C}, (b^i x))}{n} = \lambda(E-\mathscr{C}).
$$ 
Finally, since $(T_{Q,n}(x))$ never lies in $\mathscr{C}$, we have that 
$$
\lim_{n \to \infty} \frac{A_n(E, (b^i x))}{n} = \lambda(E-\mathscr{C}) = \lambda(E),
$$ 
so $(b^i x)$ is u.d. mod 1. Thus $x$ is normal in base $b$.
\end{itemize}

\begin{proof}[Proof of \reft{secondthrm}]
Set $Q = (g^2,g^2,g^2, \cdots)$ and let $x = .d_1 d_2 d_3 \cdots \wrt{Q}$ be a real number. If $x$ is not $Q$-normal (and therefore not $Q$-distribution normal), then we are done. If $x$ is $Q$-normal (and therefore $Q$-distribution normal), then construct a new basic sequence $P$ as follows\footnote{For example, if $g=2$ and $x = 0.132113\cdots \wrt{Q}$, then $P = (4,2,2,4,4,4,2,2, \cdots)$.}. If $d_n = g^2-1$, put $p_{n+N_n^Q(g^2-1,x)} = g$ and $p_{n+N_n^Q(g^2-1, x)+1} = g$. If $d_n \neq g^2-1$, put $p_{n+N_n^Q(g^2-1,x)} = g^2$.

Note that if $x = .e_1 e_2 e_3 \cdots \wrt{P}$, then $e_i \neq g^2-1$ for all $i$. Since $x$ is $Q$-normal (and therefore $Q$-distribution normal) $x$ is not $P$-normal as $p_n = g^2$ infinitely often and $P$ is bounded. Consider the interval $[0, 2/g)$. Note that $T_{P,i}(x) \in E$ if and only if $e_i <2$ when $p_i = g$ or $e_i < 2g$ when $p_i = g^2$. Thus 
$$
A_n(E, (T_{P,i}(x))) = \sum_{j = 0}^{2g-1} N_n^Q(j, x) + 2 N_n^Q(g^2-1, x).
$$
Since $x$ is $Q$-normal we have
$$
\lim_{n \to \infty} \frac{A_n(E,(T_{P,i}(x)))}{n} = (2g-1) \frac{1}{g^2} + \frac{2}{g^2} = \frac{2g+1}{g^2} \neq \frac{2}{g}.
$$
Thus $x$ is not $P$-distribution normal.
\end{proof}

It should be noted that \reft{secondthrm} follows immediately from Theorem 1.2.5 in P. Lafer's dissertation \cite{Lafer} but we prefer to give a constructive proof.
\section{Further questions}

\begin{problem}
For which pairs of basic sequences $(P,Q)$ does $\NQ = \N{P}$ or $\DNQ = \DN{P}$?
\end{problem}

\begin{problem}
For which pairs $(Q,b)$ of a basic sequence and an integer greater than or equal to 2 does $\NQ = \N{b}$ or $\DNQ = \N{b}$?
\end{problem}

\begin{problem}
For which basic sequences $Q$ does $\N{Q} = \DN{Q}$?
\end{problem}

\begin{problem}
For which basic sequences $Q$ does there is exist a number-theoretic transformation $T$ with $\N{T} = \DN{Q}$?
\end{problem}

\begin{problem}
Is there a basic sequence $Q$ such that $\DN{Q}$ is the set of real numbers normal with respect to the regular continued fraction expansion?
\end{problem}

\bibliographystyle{amsplain}

\providecommand{\bysame}{\leavevmode\hbox to3em{\hrulefill}\thinspace}
\providecommand{\MR}{\relax\ifhmode\unskip\space\fi MR }
\providecommand{\MRhref}[2]{%
  \href{http://www.ams.org/mathscinet-getitem?mr=#1}{#2}
}
\providecommand{\href}[2]{#2}

\end{document}